\documentclass[letterpaper, 11pt,  reqno]{amsart}

\usepackage{amsmath,amssymb,amscd,amsthm,amsxtra, esint}
\usepackage{tikz} 

\usepackage{comment}
\usepackage{mathtools}
\usepackage{color}
\usepackage[implicit=true]{hyperref}
\usepackage{scalerel} 

\usepackage{cases}

\usepackage[margin=1.2in,marginparwidth=1.5cm, marginparsep=0.5cm]{geometry}

\allowdisplaybreaks[2]

\usepackage{color}

\definecolor{gr}{rgb}   {0.,   0.69,   0.23 }
\definecolor{bl}{rgb}   {0.,   0.5,   1. }
\definecolor{mg}{rgb}   {0.85,  0.,    0.85}
\definecolor{yl}{rgb}   {0.8,  0.7,   0.}
\definecolor{or}{rgb}  {0.7,0.2,0.2}

\usetikzlibrary{shapes.misc}
\usetikzlibrary{shapes.symbols}
\usetikzlibrary{shapes.geometric}

\tikzset{
	ddot/.style={circle,fill=white,draw=black,inner sep=0pt,minimum size=0.8mm},
	>=stealth,
	}

\tikzset{
	ddot2/.style={circle,fill=black,draw=black,inner sep=0pt,minimum size=0.8mm},
	>=stealth,
	}

\newtheorem{theorem}{Theorem} [section]

\newtheorem{lemma}[theorem]{Lemma}
\newtheorem{proposition}[theorem]{Proposition}
\newtheorem{remark}[theorem]{Remark}

\DeclareMathOperator*{\intt}{\int}
\DeclareMathOperator*{\iintt}{\iint}

\DeclareMathOperator{\MAX}{MAX}

\newcommand{\noi}{\noindent}
\newcommand{\Z}{\mathbb{Z}}
\newcommand{\R}{\mathbb{R}}

\newcommand{\T}{\mathbb{T}}

\newcommand{\Pf}{\mathfrak{P}}
\newcommand{\Sf}{\mathfrak{S}}

\let\Re=\undefined\DeclareMathOperator*{\Re}{Re}
\let\Im=\undefined\DeclareMathOperator*{\Im}{Im}

\newcommand{\F}{\mathcal{F}}

\newcommand{\al}{\alpha}

\newcommand{\dl}{\delta}

\newcommand{\Dl}{\Delta}
\newcommand{\eps}{\varepsilon}

\newcommand{\G}{\Gamma}

\newcommand{\ft}{\widehat}

\newcommand{\wt}{\widetilde}
\newcommand{\cj}{\overline}

\newcommand{\dt}{\partial_t}

\newcommand{\ta}{\theta}

\renewcommand{\l}{\ell}

\newcommand{\les}{\lesssim}
\newcommand{\ges}{\gtrsim}

\newcommand{\jb}[1]
{\langle #1 \rangle}

\newcommand{\ind}{\mathbf 1}

\newcommand{\NN}{\mathcal{N}}

\newcommand{\N}{\mathbb{N}}

\newtheorem*{ackno}{Acknowledgements}

\numberwithin{equation}{section}
\numberwithin{theorem}{section}

\makeatletter
\@namedef{subjclassname@2020}{%
  \textup{2020} Mathematics Subject Classification}
\makeatother

\begin{document}
\baselineskip = 13.5pt

\title[Sharp LWP of the 2-$d$ quadratic NLS]
{Sharp local well-posedness of the two-dimensional periodic nonlinear Schr\"odinger equation with a quadratic nonlinearity $|u|^2$}

\author[R.~Liu and  T.~Oh]
{Ruoyuan Liu and  Tadahiro Oh}

\address{
Ruoyuan Liu,  School of Mathematics\\
The University of Edinburgh\\
and The Maxwell Institute for the Mathematical Sciences\\
James Clerk Maxwell Building\\
The King's Buildings\\
Peter Guthrie Tait Road\\
Edinburgh\\ 
EH9 3FD\\
 United Kingdom}

\email{ruoyuan.liu@ed.ac.uk}


\address{
Tadahiro Oh, School of Mathematics\\
The University of Edinburgh\\
and The Maxwell Institute for the Mathematical Sciences\\
James Clerk Maxwell Building\\
The King's Buildings\\
Peter Guthrie Tait Road\\
Edinburgh\\ 
EH9 3FD\\
 United Kingdom}

\email{hiro.oh@ed.ac.uk}

\subjclass[2020]{35Q55}
\keywords{nonlinear Schr\"odinger equation; well-posedness}
%


\begin{abstract}
We study the  nonlinear Schr\"odinger equation (NLS)
with the quadratic nonlinearity~$|u|^2$, posed on the two-dimensional torus $\T^2$.
While the relevant $L^3$-Strichartz estimate is known only with a derivative loss, 
we prove local well-posedness of the quadratic NLS in $L^2(\T^2)$, 
thus resolving an open problem of thirty years since Bourgain (1993).
In view of ill-posedness in negative Sobolev spaces, 
this result is sharp.
We establish a crucial bilinear estimate by separately
studying the non-resonant and nearly resonant cases.
As a corollary, we obtain 
a tri-linear version of the $L^3$-Strichartz estimate
without any derivative loss.

\end{abstract}

\maketitle

%

\section{Introduction}
\label{SEC:intro}

We consider the following Cauchy problem for the quadratic nonlinear Schr\"odinger equation (NLS) on the two-dimensional torus $\T^2 = (\R / \Z)^2$:
\begin{equation}
\begin{cases}
i \dt u + \Dl u = |u|^2 \\
u|_{t = 0} = u_0.
\end{cases}
\label{qNLS}
\end{equation}

\noi
Our main goal is to prove local well-posedness of \eqref{qNLS} in the low regularity setting.
Over the  last forty years, the Cauchy problems of NLS with various types of nonlinearities 
have been studied extensively.
The following NLS with a gauge-invariant nonlinearity:
\begin{equation}
i \dt u + \Dl u =  \pm |u|^{p-1} u 
\label{NLS}
\end{equation}

\noi
is the most well-known and well-studied example.
In the case of the $d$-dimensional torus $\T^d = (\R/\Z)^d$, Bourgain  \cite{B93} 
introduced the Fourier restriction norm method (via the $X^{s, b}$-spaces)
and proved local well-posedness of \eqref{NLS}
in  the low regularity setting.
In particular, when $d = 2$, he proved the following 
$L^4$-Strichartz estimate (with a derivative loss):
\begin{align}
\| e^{it \Dl}  f\|_{L^4([0, 1]; L^4(\T^2))} \les \|f\|_{H^s}
\label{Str1}
\end{align}

\noi
for any $s> 0$, which allowed him to prove
local well-posedness of the cubic NLS, \eqref{NLS} with $p = 3$, 
in $H^s(\T^2)$, for any $s> 0$. We point out that the estimate \eqref{Str1}
fails when $s = 0$ and 
the well-posedness issue of the cubic NLS
in the critical space $L^2(\T^2)$ remains
a challenging open problem; see \cite{B93, TTz, Kishi}.

Let us now turn to the quadratic NLS \eqref{qNLS}.
By interpolating \eqref{Str1} with the trivial $L^2$-bound, 
we obtain the $L^3$-Strichartz estimate with a derivative loss.
Then by proceeding as in~\cite{B93}, 
we immediately obtain local well-posedness
of \eqref{qNLS} in $H^s(\T^2)$ for any $s > 0$.
Due to the derivative loss in the Strichartz estimate, 
however, its well-posedness in $L^2(\T^2)$ has remained open 
for the last thirty years.
In this paper, we prove that \eqref{qNLS} is
indeed locally well-posed in $L^2(\T^2)$.

\begin{theorem}\label{THM:LWP}
The quadratic NLS \eqref{qNLS} is locally well-posed in $L^2(\T^2)$. More precisely, given any $u_0 \in L^2(\T^2)$, there exists $T = T(\| u_0 \|_{L^2}) > 0$ and a unique solution $u
\in C([-T, T]; L^2(\T^2))$ to \eqref{qNLS} with $u|_{t=0} = u_0$.
\end{theorem}

Our proof is based on the Fourier restriction norm method and, as such, 
 the uniqueness holds only in (the local-in-time version of) the relevant $X^{s,b}$-space.

A few remarks are in order.
In \cite{Kishi3}, Kishimoto proved ill-posedness of \eqref{qNLS} in $H^s(\T^2)$ for $s < 0$
and thus
Theorem \ref{THM:LWP} is sharp.  See Remark \ref{REM:1} below.
Theorem \ref{THM:LWP} is also sharp in the sense that 
local-in-time solutions constructed in Theorem \ref{THM:LWP}
can not be in general extended globally in time.
Indeed,  the second author \cite{O1} proved a finite-time blowup result
for \eqref{qNLS}. See also~\cite{FO}.
This argument can be easily extended to $L^2(\T^2)$, yielding the following proposition.

\begin{proposition}
Let $s \geq 0$ and $u_0 \in H^s(\T^2)$. If the initial data $u_0$ satisfies
\[ \Im \int_{\T^2} u_0 dx < 0 \quad \text{or} \quad \Re \int_{\T^2} u_0 dx \neq 0,  \]
then the forward maximal existence time $T^*$ of the solution $u$ to \eqref{qNLS} 
with $u|_{t = 0} = u_0$ is finite
and we have  $\liminf_{t \nearrow T^*} \| u (t) \|_{H^s} = \infty$.
\end{proposition}

See \cite{FO} on the lifespan of solutions to \eqref{qNLS}.
We point out that, in the one-dimensional case, Fujiwara and Georgiev \cite{FG}
recently proved the criterion for global existence.
Consider the quadratic NLS \eqref{qNLS} on the one-dimensional torus $\T$.
Then, there exists a global $L^2$-solution $u$ to \eqref{qNLS} on $\T$
if and only if $\Re u_0 = 0$ and $\Im u_0 = \mu$
for some $\mu \geq 0$.  In particular, any global $L^2$-solution
is necessarily constant in space. It would be of interest to investigate this issue
in the two-dimensional case.

As mentioned above, our proof of Theorem \ref{THM:LWP} is based
on the Fourier restriction norm method.
More precisely, Theorem \ref{THM:LWP} follows
from a standard contraction argument, once we proof the following bilinear estimate.

\begin{proposition}\label{PROP:bilin}
Let $0 < T \leq 1$. Let $\dl_1 > \dl_2 > 0$ be sufficiently small.
Then,  for $s\ge 0 $, we have
\begin{align} \| u \cj{v} \|_{X_T^{s,-\frac 12 + \dl_1}} \les \| u \|_{X_T^{s, \frac 12 + \dl_2}} \| v \|_{X_T^{s, \frac 12 + \dl_2}}. 
\label{bilin1}
\end{align}
\end{proposition}

Here, $X^{s, b}_T$ denotes the local-in-time version of the $X^{s, b}$-space.
See Section \ref{SEC:2} for the definition.
As mentioned above, the $L^3$-Strichartz comes with a derivative loss
and hence can not be used directly to prove \eqref{bilin1}.
We instead separate the proof into two cases:
(i)~non-resonant interaction and (ii) nearly resonant interaction.
In the non-resonant case, thanks to the gain of derivative via
multilinear dispersion, 
we can make up for the loss of the derivative in the $L^3$-Strichartz estimate.
In the nearly resonant case, we notice that the angle
between the second incoming wave and the outgoing wave
is almost perpendicular.
This angular restriction allows us to prove the estimate without any derivative loss.
See also \cite{Tao4, CKSTT08}.
See Section~\ref{SEC:bilin} for details.

%
%

We conclude this introduction by several remarks.

\begin{remark}\label{REM:1} \rm
(i) Unlike \cite{B93}, 
our argument does not rely on intricate number theoretic properties
and thus Theorem \ref{THM:LWP} also holds on a general torus
$\T^2_{\pmb{\al}} = (\R/\al_1 \Z) \times (\R/\al_2 \Z) $
for any ratio $\pmb{\al} = (\al_1, \al_2)$ with  $\al_1, \al_2 > 0$. 
   This is essentially due to the fact that
   the key bilinear Strichartz estimate
   and the counting lemma (see Lemmas \ref{LEM:bilin} and \ref{LEM:count} below) hold on a general flat torus (see Lemmas 2.5 and 2.9 in~\cite{Kish13}) and the $L^4$-Strichartz estimate in Lemma \ref{LEM:L4} also holds on  a general torus with an $\eps$ derivative loss (see Theorem 2.4 in \cite{BD15}).
We also mention
\cite{B07, CW, GOW, BD15, KV, DGG}
for further discussions on the Strichartz estimates
and well-posedness of NLS on irrational tori.

\smallskip

\noi
(ii) The bilinear estimate \eqref{bilin1} also holds
for $u v$ and $\cj{u} \cj{v}$ in place of $u \cj{v}$ on the left-hand side.
Indeed,  for $\cj{u} \cj{v}$, Gr\"unrock \cite{Gr}
proved the corresponding bilinear estimate for $s > - \frac 12$.
As for $uv$, a slight modification of the proof of Proposition \ref{PROP:bilin}
yields the corresponding bilinear estimate for $s \ge 0$.
See Remark \ref{REM:bilin}.
These bilinear estimates yield local well-posedness
in the corresponding ranges.

The quadratic NLS on $\T^2$ is critical
in $H^{-1}(\T^2)$ with respect to the scaling symmetry.
Nonetheless, Kishimoto \cite{Kishi3} proved ill-posedness of the 
quadratic NLS on $\T^2$ in $H^s(\T^2)$
for (a) $s < 0 $ with the nonlinearity $\NN(u) =  |u|^2$
and  
(b) $s \le -1$ with $\NN(u) = u^2$ or $\cj{u}^2$.
As mentioned above, the well-posedness theory of 
the quadratic NLS~\eqref{qNLS} with the nonlinearity $|u|^2$
is now complete.
On the other hand, as for the nonlinearity $u^2$ or $\cj u^2$, 
there is still a gap between the well-posedness and ill-posedness regularities.

\smallskip

\noi
(iii) 
It is conjectured in \cite{B93} that, on $\T^2$, the $L^p$-Strichartz estimate holds
without any derivative loss as long as $  p < 4$.
At this point, the $L^p$-Strichartz estimate on $\T^2$
for $2< p < 4$ is known to hold with a slight loss of derivative
and  this conjecture remains open.
By considering a multilinear version of the Strichartz estimate, however, 
we obtain 
the following tri-linear version of the $L^3$-Strichartz estimate without any derivative loss:
\begin{align}
\bigg|\int_0^1 \int_{\T^2} \prod_{j = 1}^3 \big(e^{it \Dl} \phi_j\big)^*
dx dt \bigg|
\les \prod_{j = 1}^3 \|\phi_j \|_{L^2},
\label{tri1}
\end{align}

\noi
where $u^*$ denotes $ u$ or $\cj u$.
The estimate \eqref{tri1} follows easily from 
the bilinear estimate \eqref{bilin1} (also for $uv$ and $\cj u \cj v$)
and the duality.
We point out that it is crucial that 
we have a product structure on the left-hand side of \eqref{tri1}.

%
%
%
%
%
%
%

\smallskip

\noi
(iv) Lastly, let us consider the following quadratic NLS on $\T^2$
with a gauge-invariant nonlinearity:
\begin{equation}
i \dt u + \Dl u = \pm |u| u .
\label{NLS9}
\end{equation}

\noi
This equation is of particular interest in view of the mass and energy conservations.
In particular, local well-posedness in $L^2(\T^2)$
together with the mass conservation would imply
global well-posedness in the same space.

When $s> 0$, local well-posedness of \eqref{NLS9} in $H^s(\T^2)$
easily 
follows from the $L^4$-Strichartz estimate~\eqref{Str1}
(more precisely, Lemma \ref{LEM:L4} below).
When $s = 0$, however, the non-algebraic nature of the nonlinearity
makes the local well-posedness problem of \eqref{NLS9} in $L^2(\T^2)$
rather challenging.
For example, the tri-linear estimate \eqref{tri1}
is not useful to study \eqref{NLS9}
and,  moreover, multilinear analysis via the Fourier restriction
norm method (such as that presented in this paper)
is not applicable due to the presence of 
$|u|$.
While there are well-posedness results
\cite{OOP, Lee}
on the periodic  NLS with non-algebraic gauge-invariant nonlinearities, 
one would need a much more intricate argument
to prove local well-posedness of \eqref{NLS9} in $L^2(\T^2)$.

\end{remark}

\section{Notations and preliminary lemmas}
\label{SEC:2}


\subsection{Notations}
For a spacetime function $u$ defined on $\R \times \T^2$, we write $\mathcal{F}_{t,x} u$ or $\ft u$ to denote the spacetime Fourier transform of $u$.
In the following, we drop the inessential factor of $2\pi$.
We  also set $\jb{\,\cdot\,} = (1 + |\cdot|^2)^\frac{1}{2}$.

Given a dyadic number $N  \ge 1$, we let $P_N$ be the spatial  frequency projector onto the
 frequencies 
\[ \Pf_N := \big\{ (\tau, n)\in \R\times \Z^2 : \tfrac{N}{2} < |n| \leq N \big\}.\] 
Also, given a dyadic number $L \geq 1$, we define $Q_L$ to be the 
modulation projector onto the space-time frequencies 
\[\Sf_L :=\big \{ (\tau, n)\in \R\times\Z^2: \tfrac{L}{2} < \big| \tau + |n|^2 \big| \leq L \big\}.\]
 For brevity, we also set  $P_{N,L} = P_N Q_L$.

In what follows, $N$ and $L$ (possibly with subscripts) always denote dyadic numbers $\geq 1$. We write $\overline{N}_{ij\cdots} := \max(N_i, N_j, \dots )$ and $\underline{N}_{ij\cdots} := \min ( N_i, N_j, \dots )$. We also write $N_{\text{max}} := \overline{N}_{012}$, $N_{\text{min}} := \underline{N}_{012}$, $L_{\text{max}} := \overline{L}_{012}$, $L_{\text{min}} := \underline{L}_{012}$, and $L_{\text{med}} := L_0 L_1 L_2 / L_{\text{min}} L_{\text{max}}$.

We use $A \les B$ to denote $A \leq CB$ for some constant $C > 0$, and we write $A \sim B$ to denote $A \les B$ and $B \les A$. We also write $A \ll B$ if $A \leq cB$ for some sufficiently small $c > 0$.
We may use subscripts to denote dependence on external parameters; for example,
 $A\les_{p, q} B$ means $A\le C(p, q) B$,
 where the constant $C(p, q)$ depends on parameters $p$ and $q$.
 In addition, we use $a+$ (and $a-$, respectively) to denote $a+\eps$ (and $a-\eps$, respectively)
for arbitrarily small $\eps > 0$.

\subsection{Preliminary lemmas}
In this subsection, we recall some useful lemmas. 
%
We first recall the following bilinear  Strichartz estimate. For the proof, see Lemma 2.5 in~\cite{Kish13}.
\begin{lemma}\label{LEM:bilin}
Let $N_j, L_j \geq 1$,  $j = 0,1,2$,  be dyadic numbers. Suppose that $u_1, u_2 \in L^2(\R \times \T^2)$ satisfy
\[ \textup{supp } \ft{u_1} \subset \Pf_{N_1} \cap \Sf_{L_1}
\quad\text{and}\quad \textup{supp } \ft{u_2} \subset \Pf_{N_2} \cap \Sf_{L_2}.  \]
Then, we have
\[ \|  \F_{t, x}(u_1 \cj{u_2}) \|_{L_\tau^2 \ell_n^2(\Pf_{N_0})} 
\les \underline{L}_{12}^{\frac 12} \bigg( \frac{\overline{L}_{12}}{N_0} + 1 \bigg)^{\frac 12} N_{\textup{min}}^{\frac 12} \| u_1 \|_{L_{t,x}^2} \| u_2 \|_{L_{t,x}^2}. \]
\end{lemma}

We also recall the following counting lemma. For the proof, see Lemma 2.9\,(ii) in \cite{Kish13}, which was stated in general dimension $d \geq 2$ but we only need the $d = 2$ case.

\begin{lemma}\label{LEM:count}
Let $N \gg 1$, $N^{-1} \leq \mu, \nu \ll N$, $M \geq 0$, and set
\[ D := \{ \xi = (\xi_1, \xi_2) \in \R^2: N \leq |\xi| \leq N + \mu, M \leq \xi_1 \leq M + \nu \}. \]
Let $\mathcal{R}$ be an arbitrary rotation operator on $\R^2$. 
Moreover, with $e_1 = ( 1,0) \in \R^2$, set 
\[ K := \Big\{ \xi \in \R^2: \frac{\al}{2} \leq \angle(\xi, e_1) \leq 2\al \Big\},\]

\noi
where $\angle(\xi, e_1) $ denotes the angle between $\xi$ and $e_1$.
Suppose that 
\[ 
\bigg( \frac{\mu + \min\{\nu, 1\}}{N} \bigg)^{\frac 12 } \ll \al \leq \frac{\pi}{4}. \]
Then, we have
\[ |\Z^2 \cap \mathcal{R}(D \cap K)| \les \max\{ \nu, 1 \} \big(\al^{-1}(\mu + \min\{ \nu, 1 \}) + 1\big). \]
\end{lemma}

\subsection{Fourier restriction norm method}
In studying well-posedness of NLS on $\T^d$, Bourgain \cite{B93}
introduced the Fourier restriction norm method (see also \cite{KM93}), 
utilizing the following $X^{s, b}$-spaces.
 Given $s, b \in \R$, 
 we define the space $X^{s,b}(\R \times \T^2)$ 
 to be  the completion of functions that are $C^\infty$ in space and Schwartz in time with respect to the following norm:
\begin{align}
 \| u \|_{X^{s,b}(\R \times \T^2)} := \big\| \jb{n}^s \jb{\tau + |n|^2}^b \ft u(\tau, n) \big\|_{ L_\tau^2 \ell_n^2(\R \times \Z^2) }. 
\label{Xsb}
\end{align}

\noi
For $T > 0$, we define the space $X_T^{s,b}$ to be the restriction of 
the $X^{s, b}$-space onto the time interval 
$[-T, T]$ via the norm:
\begin{equation}
\| u \|_{X_T^{s,b}} := \inf \big\{ \| v \|_{X^{s,b}}: v|_{[-T, T]} = u \big\}.
\label{XT}
\end{equation}

\noi
Note that $X^{s, b}_T$ is complete.
Given any $s \in \R$ and $b > \frac 12$, we have $X_T^{s,b} \subset C([-T, T]; H^s(\T^2))$.

We now recall the following linear estimates. See \cite{B93, GTV, Tao}.

\begin{lemma}\label{LEM:Xlin}
Let $s \in \R$ and $0 < T \leq 1$.

\smallskip \noi
\textup{(i)} For any $b \in \R$, we have
\[ \| e^{it\Dl} \phi \|_{X_T^{s,b}} \les_b \| \phi \|_{H^s}. \]

\smallskip \noi
\textup{(ii)} Let $-\frac 12 < b' \leq 0 \leq b \leq b'+1$. Then, we have
\[ \bigg\| \int_0^t e^{i(t-t')\Dl} F(t') dt' \bigg\|_{X_T^{s,b}} \les_{b,b'} T^{1-b+b'} \| F \|_{X_T^{s,b'}}. \]
%
\end{lemma}

Lastly,  we record the following $L^4$-Strichartz estimate on $\T^2$. 

\begin{lemma}\label{LEM:L4}
Let $s > 0$ and $b > \frac 12$. Then, we have
\begin{align} \| u \|_{L_{t,x}^4([0,1] \times \T^2)} \les_{s, b} \| u \|_{X^{s, b}}. 
\label{L42}
\end{align}
\end{lemma}

For the proof, see \cite{B93, B95}.  We point out that the $L^4$-Strichartz estimate 
\eqref{L42} does not hold when $s=0$.  See \cite{TTz}.
As mentioned in the introduction, the $L^4$-Strichartz estimate \eqref{Str1}
with an $\eps$-loss of derivative
holds on a general flat torus
$\T^2_{\pmb{\al}} = (\R/\al_1 \Z) \times (\R/\al_2 \Z) $
for any ratio $\pmb{\al} = (\al_1, \al_2)$ with  $\al_1, \al_2 > 0$;
see  Theorem~2.4 in \cite{BD15}.
Then, the estimate \eqref{L42}
on a general flat torus $\T^2_{\pmb{\al}}$ follows from a standard transference principle (see Lemma 2.9 in \cite{Tao}).

\section{Proof of Proposition \ref{PROP:bilin}}
\label{SEC:bilin}

In this section, we  present the proof of  Proposition~\ref{PROP:bilin}.
By the triangle inequality $\jb{n}^s \les \jb{n_1}^s \jb{n_2}^s$
under $n = n_1 + n_2$ for $s \ge 0$, 
 our goal is to prove  the following bilinear estimate (with $s = 0$):
\begin{align}
 \| u \cj{v} \|_{X_T^{0,-\frac 12 + \dl_1}} \les \| u \|_{X_T^{0, \frac 12 + \dl_2}} 
 \| v \|_{X_T^{0, \frac 12 + \dl_2}}, 
\label{bilin1a}
\end{align}

\noi
where $0 < T \leq 1$ and $\dl_1, \dl_2 > 0$ are sufficiently small, 
satisfying  $\dl_2 < \dl_1 < 2\dl_2$. 
In view of the definitions \eqref{Xsb} and \eqref{XT} of the $X^{s, b}$-spaces, 
the bilinear estimate \eqref{bilin1a} follows
once we prove
\begin{equation}
\begin{split}
\bigg\| & \sum_{\substack{n_1\in \Z^2\\ n = n_1 - n_2}} 
\intt_{\tau = \tau_1 - \tau_2} \frac{\ft u(\tau_1, n_1) \cj{\ft v(\tau_2, n_2)}}{\jb{\tau_1 + |n_1|^2}^{\frac 12 + \dl_2} \jb{\tau_2 + |n_2|^2}^{\frac 12 + \dl_2} \jb{\tau + |n|^2}^{\frac 12 - \dl_1}} d\tau_1 \bigg\|_{L_\tau^2 \ell_n^2} \\
& \les \| u \|_{L_{t,x}^2} \| v \|_{L_{t,x}^2}.
\end{split}
\label{bilin2}
\end{equation}

\noi
By duality, the estimate \eqref{bilin2} 
follows once we prove the following estimate:\
\begin{equation}
\begin{split}
\bigg| &\sum_{\substack{n, n_1\in \Z^2\\ n = n_1 - n_2}} 
\iintt_{\tau = \tau_1 - \tau_2} \frac{\ft u (\tau_1, n_1) \cj{\ft v (\tau_2, n_2) }\cj{\ft w (\tau, n)}}{\jb{\tau_1 + |n_1|^2}^{\frac 12 + \dl_2} \jb{\tau_2 + |n_2|^2}^{\frac 12 + \dl_2} \jb{\tau + |n|^2}^{\frac 12 - \dl_1}}
d\tau_1 d\tau \bigg| \\
& \les \| u \|_{L_{t,x}^2} \| v \|_{L_{t,x}^2} \| w \|_{L_{t,x}^2}.
\end{split}
\label{bilin3}
\end{equation}

\noi
We first note that if  $n = 0$, $n_1 = 0$, or $n_2 = 0$, 
then the estimate \eqref{bilin3} follows easily from  the Cauchy-Schwarz inequality, 
provided that $\dl_1 > 0$ is sufficiently small. 
 Hence,  for the remaining of this section, we assume that $n \neq 0$, $n_1 \neq 0$, and $n_2 \neq 0$.

Before proceeding further, we recall
the following key algebraic relation:
\begin{equation}
(\tau + |n|^2) - (\tau_1 + |n_1|^2) + (\tau_2 + |n_2|^2) = -2 n \cdot n_2, 
\label{phase}
\end{equation}

\noi
where  $n = n_1 - n_2$ and $\tau = \tau_1 - \tau_2$.
The main difficulty in proving \eqref{bilin3} (namely, the bilinear estimate \eqref{bilin1} with $s = 0$)
comes from the failure of 
the $L^4$-Strichartz estimate (without a derivative loss).
In order to overcome this difficulty, we separately estimate the contributions coming
from 
(i) non-resonant case: $|n \cdot n_2| \ges |n|^\eps |n_2|^\eps$
and (ii) nearly resonant case: $|n \cdot n_2| \ll |n|^\eps |n_2|^\eps$
for some small $ \eps > 0$.
In the non-resonant case, 
we can use the multilinear dispersion
to make up for the derivative loss in the $L^4$-Strichartz estimate (Lemma~\ref{LEM:L4}).
In the nearly resonant case, by noting that the frequencies $n$ and $n_2$
are almost perpendicular, 
we make use of this angular restriction to prove
a multilinear estimate without a derivative loss
in the spirit of \cite{Tao4, CKSTT08}.

\subsection{Non-resonant interaction}
In this subsection, we consider the non-resonant case:
\begin{align}
|n \cdot n_2| \ges |n|^\eps |n_2|^\eps
\label{nr1}
\end{align}

\noi
 for some small $\eps > 0$ sufficiently small.
From \eqref{phase} and \eqref{nr1}, we have
\begin{align}
\begin{split}
\MAX :\! & =  \max \big( \jb{\tau + |n|^2}, \jb{\tau_1 + |n_1|^2}, \jb{\tau_2 + |n_2|^2} \big)\\
& \ges \jb{n \cdot n_2}
\ges \jb{n}^\eps \jb{n_2}^\eps.
\end{split}
\label{nr2}
\end{align}
We then consider the following three cases.
In the following, we set   $0 < \dl_1 < \frac 12 $.

\medskip \noi
$\bullet$ \textbf{Case 1:} $\text{MAX} = \jb{\tau + |n|^2}$. 

In this case, we directly prove \eqref{bilin2}.
Then, with  $\eps_1 = \frac \eps 2(\frac 12 - \dl_1)> 0$, 
it follows from \eqref{nr2} that 
\begin{equation}
\jb{n_1}^{\eps_1} \jb{n_2}^{\eps_1} = \jb{n+n_2}^{\eps_1} \jb{n_2}^{\eps_1} \les \jb{n}^{2\eps_1} \jb{n_2}^{2\eps_1} \les \jb{\tau + |n|^2}^{\frac 12 - \dl_1}.
\label{tn0}
\end{equation}

\noi
Let $\ft U(\tau_1, n_1) = \jb{n_1}^{-\eps_1} \jb{\tau_1 + |n_1|^2}^{-\frac 12 - \dl_2} \ft u(\tau_1, n_1)$
and
$\ft V(\tau_2, n_2) = \jb{n_2}^{-\eps_1} \jb{\tau_2 + |n_2|^2}^{-\frac 12 - \dl_2} \ft v(\tau_2, n_2)$.
Then, by \eqref{tn0}, Plancherel's theorem, H\"older's inequality, and the $L^4$-Strichartz estimate (Lemma \ref{LEM:L4}), we have
\begin{align*}
\text{LHS of \eqref{bilin2}} 
&\les \| U \cj V\|_{L^2_{t, x}}
\leq \| U \|_{L^4_{t, x}}\|  V\|_{L^4_{t, x}}
\les \|U\|_{X^{\eps_1, \frac 12 + \dl_2}}
\|V\|_{X^{\eps_1, \frac 12 + \dl_2}}\\
&= \| u \|_{L_{t,x}^2} \| v \|_{L_{t,x}^2}
\end{align*}

\noi
for any $\dl_2 > 0$.

\medskip \noi
$\bullet$ \textbf{Case 2:} $\text{MAX} = \jb{\tau_1 + |n_1|^2}$. 

In this case, we prove \eqref{bilin3}. 
With  $\eps_2 = \eps(\frac 12 - \dl_1)> 0$, 
it follows from \eqref{nr2} that 
\begin{equation}
\jb{n}^{\eps_2} \jb{n_2}^{\eps_2} \les \jb{\tau_1 + |n_1|^2}^{\frac 12 - \dl_1}.
\label{tn1}
\end{equation}

\noi
Then, letting
$\ft V(\tau_2, n_2) = \jb{n_2}^{-\eps_2} \jb{\tau_2 + |n_2|^2}^{-\frac 12 - \dl_2} \ft v(\tau_2, n_2)$
and 
$\ft W(\tau, n) = \jb{n}^{-\eps_2} \jb{\tau + |n|^2}^{-\frac 12 - \dl_2} \ft w(\tau, n)$, 
it follows from 
 the Cauchy-Schwarz inequality in $\tau_1$ and $n_1$, \eqref{tn1}, Plancherel's theorem, H\"older's inequality, and the $L^4$-Strichartz estimate (Lemma \ref{LEM:L4}) that 
\begin{align*}
 \text{LHS of \eqref{bilin3}} 
&\les \|u\|_{L^2_{t, x}} 
\| V \|_{L^4_{t, x}}\|W\|_{L^4_{t, x}}\\
&\les \|u\|_{L^2_{t, x}} 
\| V \|_{X^{\eps_2, \frac 12 +\dl_2}} \|W\|_{X^{\eps_2, \frac 12 +\dl_2}} \\
&= \| u \|_{L_{t,x}^2} \| v \|_{L_{t,x}^2} \| w \|_{L_{t,x}^2}
\end{align*}

\noi
for any $\dl_2 > 0$.

\medskip
\noi
$\bullet$ \textbf{Case 3:} $\text{MAX} = \jb{\tau_2 + |n_2|^2}$.

This case follows from proceeding as in Case 2
and thus we omit details.

\subsection{Resonant interaction}
\label{SUBSEC:res}
We now consider the resonant case $|n \cdot n_2| \ll |n|^\eps |n_2|^\eps$.
With $\ta = \ta (n, n_2) = \frac{1}{|n|^{1-\eps} |n_2|^{1-\eps}}$, we rewrite this condition as
\begin{align*}
|n \cdot n_2| \ll \ta |n| |n_2|.
\end{align*}

\noi
In the following, we divide the argument  into three subcases, depending on the sizes
of $n$ and~$n_2$.

\medskip \noi
$\bullet$ \textbf{Case 1:} $|n|^{\frac 12} \les |n_2| \les |n|^2$.

In this case, we follow the idea from the proof of Proposition 10.1 in \cite{Tao4} (see also \cite{CKSTT08}).

\smallskip \noi
$\circ$ \textbf{Subcase 1.a:} $|n| \sim |n_2|$.

By applying the dyadic decompositions to the frequencies $n$ and $n_2$, we have
\begin{equation}
\begin{split}
\text{LHS of } \eqref{bilin3} 
&\les \sum_{\substack{N \sim N_2 \geq 1\\\text{dyadic}} }
\bigg| \sum_{\substack{n, n_1\in \Z^2\\ n = n_1 - n_2}} \iintt_{\tau = \tau_1 - \tau_2} 
\ind_{|\cos \angle(n, n_2)| \ll \ta} \\
&\quad \times \frac{\ft u (\tau_1, n_1) \cj{\ft{P_{N_2}v} (\tau_2, n_2)} \cj{\ft{P_N w} (\tau, n)}}{\jb{\tau_1 + |n_1|^2}^{\frac 12 + \dl_2} \jb{\tau_2 + |n_2|^2}^{\frac 12 + \dl_2} \jb{\tau + |n|^2}^{\frac 12 - \dl_1}} d\tau_1 d\tau \bigg|.
\end{split}
\label{bilin_a1}
\end{equation}

\noi
Note that when $|n| \sim |n_2| \sim N$, we have $\ta = \frac{1}{|n|^{1 - \eps} |n_2|^{1 - \eps}} \sim N^{-2 + 2 \eps}$. We now restrict $n$ and $n_2$ to the following angular sectors:
\begin{align}
\begin{split}
A_\ell : \!& = \big\{ n \in \Z^2: |n| \sim N, \,  \arg(n) = \ell N^{-2 + 2 \eps} + O(N^{-2 + 2 \eps}) \big\}, \\
A_{\ell_2} :\!&= \big\{ n_2 \in \Z^2: |n_2| \sim N_2, \, \arg(n_2) = \ell_2 N^{-2 + 2 \eps} + O(N^{-2 + 2 \eps}) \big\},
\end{split}
\label{A0}
\end{align}

\noi
where $\ell, \ell_2 \in I_N: = [1, 2\pi N^{2 - \eps}] \cap \Z$.
Since $|\cos \angle (n, n_2)| \ll \ta \sim N^{-2 + 2\eps}$, we have $|\arg(n) - \arg(n_2)| = \frac{\pi}{2} + O(N^{-2 + 2\eps})$ or $|\arg(n) - \arg(n_2)| = \frac{3\pi}{2} + O(N^{-2 + 2\eps})$. This means that for each fixed $\ell$, there exists a set  $L_2(\l)$ of size $O(1)$ such that 
$\ind_{A_\ell} (n)\cdot \ind_{A_{\ell_2}}(n_2)= 0$
unless $\l_2 \in L_2(\l)$.

Continuing with the right-hand side of \eqref{bilin_a1}, 
we insert the angular restrictions \eqref{A0}
and set 
\[ \ft v_{N_2, \l_2}(\tau_2, n_2) =  \ind_{A_{\ell_2}}(n_2)  \cdot \ft{P_{N_2}v} (\tau_2, n_2)
\quad \text{and}\quad
\ft w_{N, \l} (\tau, n) = \ind_{A_\ell}(n) \cdot \ft{P_N w} (\tau, n). \]

\noi
 Then, by 
applying  the Cauchy-Schwarz inequality in $\tau_1$ and $n_1$, 
then  in $\tau$ and $n$, and H\"older's inequality in $\tau_1$ and $n_1$, 
we obtain
\begin{equation}
\begin{split}
\text{LHS of } \eqref{bilin3}
 &\les \sum_{\substack{N \sim N_2 \ge 1\\\text{dyadic}}} 
 \sum_{\ell  \in I_N} 
 \sum_{ \ell_2 \in L_2(\ell)} 
 \bigg| \sum_{\substack{n, n_1\in \Z^2\\ n = n_1 - n_2}} 
 \iintt_{\tau = \tau_1 - \tau_2}  \ind_{|\cos \angle (n, n_2)| \ll N^{-2 + 2\eps}} \\
&\quad \times \frac{ \ft u (\tau_1, n_1) \cj{\ft v_{N_2, \l_2} (\tau_2, n_2) }
  \cj{\ft  w_{N, \l} (\tau, n) }}
  {\jb{\tau_1 + |n_1|^2}^{\frac 12 + \dl_2} \jb{\tau_2 + |n_2|^2}^{\frac 12 + \dl_2} \jb{\tau + |n|^2}^{\frac 12 - \dl_1}} d\tau_1 d\tau \bigg| \\
&\leq  \sum_{\substack{N \sim N_2 \ge 1\\\text{dyadic}}} 
 \sum_{\ell  \in I_N} 
 \sum_{ \ell_2 \in L_2(\ell)} 
  \| \ft u \|_{L_{\tau_1}^2 \ell_{n_1}^2} \bigg\| \sum_{n\in \Z^2} \int \ind_{|\cos \angle (n, n_1 - n)| \ll N^{-2 + 2\eps}} \\
&\quad \times \frac{ \ft v_{N_2, \l_2} (\tau_1 - \tau , n_1- n) 
  \ft  w_{N, \l} (\tau, n) }{ \jb{\tau_1 - \tau + |n_1 - n|^2}^{\frac 12 + \dl_2} } d\tau \bigg\|_{L_{\tau_1}^2 \ell_{n_1}^2} \\
&\leq \| u \|_{L_{t, x}^2} \sum_{\substack{N \sim N_2 \ge 1\\\text{dyadic}}} 
 \sum_{\ell  \in I_N} 
 \sum_{ \ell_2 \in L_2(\ell)} 
  \| \ft v_{N_2, \l_2} \|_{L_{\tau_2}^2 \ell_{n_2}^2} \| \ft w_{N, \l}\|_{L_\tau^2 \ell_n^2} \\
&\quad \times \sup_{\tau_1, n_1} \bigg( \sum_n \int \frac{ \ind_{|\cos \angle (n, n_1 - n)| \ll N^{-2 + 2\eps}} \ind_{A_{\ell_2}}(n_1 - n)  \cdot \ind_{A_\ell}(n)}{\jb{\tau_1 - \tau + |n_1 - n|^2}^{1 + 2 \dl_2}} d\tau \bigg)^{\frac 12} \\
&\les \| u \|_{L_{t, x}^2 }  \sum_{\substack{N \sim N_2 \ge 1\\\text{dyadic}}} 
\sum_{\ell  \in I_N} 
 \sum_{ \ell_2 \in L_2(\ell)} 
  \| \ft v_{N_2, \l_2} \|_{L_{\tau_2}^2 \ell_{n_2}^2} \| \ft w_{N, \l}\|_{L_\tau^2 \ell_n^2}\\
& \quad \times \sup_{n_1, \l, \l_2} |\mathcal{A}_{N, \l, \l_2}(n_1)|^\frac 12 , 
\end{split}
\label{bilin_a2}
\end{equation}

\noi
where $\mathcal{A}_{N, \l, \l_2}(n_1)$ is defined by 
\[  \mathcal{A}_{N, \l, \l_2}(n_1) = \big\{ n \in \Z^2: |\cos \angle (n, n_1 - n)| \ll N^{-2 + 2\eps}, \, n \in A_\ell,  \, n_1 - n \in A_{\ell_2}\big \}.\]

In view of \eqref{A0}, 
for  fixed $n_1 \in \Z^2$ and fixed $\ell, \ell_2\in I_N$, 
we note that any $n\in \mathcal{A}_{N, \l. \l_2}$
is contained in  a rectangle with sides of length $\sim$
\[ |n| N^{-2 + 2\eps} \sim N^{-1 + 2\eps} \quad \text{and} \quad |n_1 - n| N^{-2 + 2\eps} \sim N^{-1 + 2\eps}. \]

\noi
Hence, for  $0 < \eps < \frac 12$,
we have 
\begin{align}
\sup_{n_1, \l,  \l_2} |\mathcal{A}_{N, \l, \l_2}(n_1)|\les 1.
\label{A2}
\end{align}

\noi
Therefore, 
by estimating  \eqref{bilin_a2} 
with \eqref{A2}, 
 the Cauchy-Schwarz inequality in $\ell$, and then in $N \sim N_2$ we obtain the desired inequality \eqref{bilin3}.

\smallskip \noi
$\circ$ \textbf{Subcase 1.b:} $|n_2| \ll |n| \les |n_2|^2$.

 Since $|n| \gg |n_2|$ and $n_1 = n + n_2$, we have $|n_1| \sim |n|$. 
%
Note that when  $|n| \sim N \les |n_2|^2$, we have $\ta = \frac{1}{|n|^{1 - \eps} |n_2|^{1 - \eps}} \les N^{-\frac 32 + \frac 32 \eps}$. We proceed as in Subcase 1.a by restricting $n$ and $n_2$ to the following angular sectors:
\begin{align}
\begin{split}
B_\l : \! & = \big\{ n \in \Z^2: |n| \sim N, \arg(n) = \ell N^{-\frac 32 + \frac 32 \eps} + O(N^{-\frac 32 + \frac 32 \eps}) \big\}, \\
B_{\l_2} :\! & = \big\{ n_2 \in \Z^2: N^{\frac 12} \les |n_2| \ll N, \arg(n_2) = \ell_2 N^{-\frac 32 + \frac 32 \eps} + O(N^{-\frac 32 + \frac 32 \eps}) \big\},
\end{split}
\label{B0}
\end{align}

\noi
where $\ell, \ell_2 \in I_N': = [1, 2\pi N^{\frac 32 - \frac 32\eps}] \cap \Z$.
Under $|\cos \angle (n, n_2)| \ll N^{-\frac 32 + \frac 32 \eps}$, we have $|\arg(n) - \arg(n_2)| = \frac{\pi}{2} + O(N^{-\frac 32 + \frac 32 \eps})$ or $|\arg(n) - \arg(n_2)| = \frac{3\pi}{2} + O(N^{-\frac 32 + \frac 32 \eps})$. 
Hence, 
 for each fixed $\ell$, there exists a set  $L_2'(\l)$ of size $O(1)$ such that 
$\ind_{B_\ell} (n)\cdot \ind_{B_{\ell_2}}(n_2)= 0$
%
unless $\l_2 \in L_2'(\l)$.
With a slight abuse of notation, set 
\[ \ft v_{N, \l_2}(\tau_2, n_2) =  \ind_{B_{\ell_2}}(n_2)  \cdot \ft{v} (\tau_2, n_2)
\quad \text{and}\quad
\ft w_{N, \l} (\tau, n) = \ind_{B_\ell}(n) \cdot \ft{P_N w} (\tau, n). \]

\noi
 Then, by proceeding as in \eqref{bilin_a2}, i.e.~by
applying  the Cauchy-Schwarz inequality in $\tau_1$ and $n_1$, 
then  in $\tau$ and $n$, and H\"older's inequality in $\tau_1$ and $n_1$, 
we obtain
\begin{align*}
\text{LHS of } \eqref{bilin3}
 &\les \sum_{\substack{N \sim N_1\ge 1\\\text{dyadic}}}
   \sum_{\ell  \in I_N'} 
 \sum_{ \ell_2 \in L_2'(\ell)} 
   \bigg| \sum_{\substack{n, n_1\in \Z^2\\ n = n_1 - n_2}} \iintt_{\tau = \tau_1 - \tau_2} \ind_{|\cos \angle (n, n_2)| \ll N^{-\frac 32 + \frac 32 \eps}}  \\
&\quad \times \frac{\ft{P_{N_1} u} (\tau_1, n_1) \cj{\ft v_{N, \l_2} (\tau_2, n_2)}
 \cj{\ft w_{N, \l} (\tau, n)}}{\jb{\tau_1 + |n_1|^2}^{\frac 12 + \dl_2} \jb{\tau_2 + |n_2|^2}^{\frac 12 + \dl_2} \jb{\tau + |n|^2}^{\frac 12 - \dl_1}} d\tau_1 d\tau \bigg| \\
&\leq  \sum_{\substack{N \sim N_1  \ge 1\\\text{dyadic}}}
   \sum_{\ell  \in I_N'} 
 \sum_{ \ell_2 \in L_2'(\ell)} 
\big\| \ft{P_{N_1} u} \big\|_{L_{\tau_1}^2 \ell_{n_1}^2} \bigg\| \sum_{n \in  \Z^2} \int \ind_{|\cos \angle (n, n_1 - n)| \ll N^{-\frac 32 + \frac 32 \eps}}  \\
&\quad \times \frac{\ft v_{N, \l_2}  (\tau_1 - \tau, n_1 - n)\ft w_{N, \l}(\tau, n) }{\jb{\tau_1 - \tau + |n_1 - n|^2}^{\frac 12 + \dl_2}} d\tau \bigg\|_{L_{\tau_1}^2 \ell_{n_1}^2} \\
&\leq  \sum_{\substack{N \sim N_1 \ge 1\\\text{dyadic}}}
    \sum_{\ell  \in I_N'} 
 \sum_{ \ell_2 \in L_2'(\ell)} 
  \| \ft{P_{N_1} u} \|_{L_{\tau_1}^2 \ell_{n_1}^2} \| \ft v \_{N, \l_2}  \|_{L_{\tau_2}^2 \ell_{n_2}^2} 
  \| w_{N, \l} \|_{L_\tau^2 \ell_n^2} \\
&\quad \times \sup_{\tau_1, n_1} \bigg( \sum_n \int \frac{\ind_{|\cos \angle (n, n_1 - n)| \ll N^{-\frac 32 + \frac 32 \eps}} \ind_{B_{\ell_2}} (n_1 - n)\cdot  \ind_{B_\ell} (n)}{\jb{\tau_1 - \tau + |n_1 - n|^2}^{1 + 2 \dl_2}} d\tau \bigg)^{1/2} \\
&\les  \sum_{\substack{N \sim N_1 \ge 1\\\text{dyadic}}}
    \sum_{\ell  \in I_N'} 
 \sum_{ \ell_2 \in L_2'(\ell)} 
  \| \ft{P_{N_1} u} \|_{L_{\tau_1}^2 \ell_{n_1}^2} \| \ft v_{N, \l_2} \|_{L_{\tau_2}^2 \ell_{n_2}^2} 
  \| \ft w_{N, \l_2} \|_{L_\tau^2 \ell_n^2} \\
& \quad \times \sup_{n_1} |\mathcal{B}_{N, \l, \l_2}(n_1)|^\frac 12, 
\end{align*}

\noi
where $\mathcal{B}_{N, \l, \l_2}(n_1)$ is defined by 
\[ \mathcal{B}_{N, \l, \l_2}(n_1)
 = \big \{ n \in \Z^2: |\cos \angle (n, n_1 - n)| \ll N^{-\frac 32 + \frac 32 \eps},, n \in B_\ell,,  n_1 - n \in B_{\ell_2} \big\}. \]

\noi
Then, the desired bound \eqref{bilin3} follows from 
 the Cauchy-Schwarz inequality in $\ell$ and then in  the Cauchy-Schwarz inequality in $N \sim N_1$,
 once we prove 
\begin{align}
\sup_{n_1, \l,  \l_2} |\mathcal{B}_{N, \l, \l_2}(n_1)|\les 1.
\label{B2}
\end{align}

In view of \eqref{B0}, 
for  fixed $n_1 \in \Z^2$ and fixed $\ell, \ell_2\in I_N'$, 
we note that any $n\in \mathcal{B}_{N, \l. \l_2}$
is contained in  a rectangle with sides of length $\sim$
\[ |n| N^{-\frac 32 + \frac 32 \eps} \sim N^{-\frac 12 + \frac 32 \eps} \quad \text{and} \quad |n_1 - n| N^{-\frac 32 + \frac 32 \eps} \ll N^{-\frac 12 + \frac 32 \eps}. \] 

\noi
Therefore, as long as $0 < \eps < \frac 13$, there are at most $O(1)$ many 
$n'$s in $ \mathcal{B}_{N, \l. \l_2}$, yielding \eqref{B2}.

\smallskip \noi
$\circ$ \textbf{Subcase 1.c:} $|n| \ll |n_2|\les |n|^2$.
\\ \indent
This subcase follows from Subcase 1.b by switching the roles of $(\tau, n)$ and $(\tau_2, n_2)$
and thus we omit details.

\medskip \noi
$\bullet$ \textbf{Case 2:} $|n|^2 \ll |n_2|$.

We divide the argument into two subcases, depending on the sizes
of  $N_2$ and the largest modulation $L_{\textup{max}} = \max (L_0, L_1, L_2)$. 

\medskip \noi
$\circ$ \textbf{Subcase 2.a:} $L_{\textup{max}} \ges N_2$ (high-modulation case).

In this subcase, recalling the notations in Section \ref{SEC:2}, we dyadically decompose
 the spatial frequencies and modulations of  $u$, $v$, and $w$ so that
  $\textup{supp } \ft {P_{N_1, L_1} u} \subset \Pf_{N_1} \cap \Sf_{L_1}$, 
   $\textup{supp } \ft {P_{N_2, L_2} v } \subset \Pf_{N_2} \cap \Sf_{L_2}$, and $\textup{supp } \ft{P_{N_0, L_0} w} \subset \Pf_{N_0} \cap \Sf_{L_0}$ for dyadic $N_j, L_j \geq 1$, $j = 0,1,2$. Note that we have $N_0^2 \ll N_2 \sim N_1$. Our main goal is to show the following lemma.

\begin{lemma}\label{LEM:high}
Let $N_j, L_j \geq 1$, $j = 0,1,2$, be dyadic numbers with $N_0^2 \ll N_2 \sim N_1$ and $L_{\textup{max}} \ges N_2$. Suppose that $f, g, h \in L^2(\R \times \Z^2)$ are real-valued nonnegative functions such that 
\begin{align}
 \textup{supp } f \subset \Pf_{N_1} \cap \Sf_{L_1}, \quad \textup{supp } g \subset \Pf_{N_2} \cap \Sf_{L_2}, \quad \text{and} \quad  \textup{supp } h \subset \Pf_{N_0} \cap \Sf_{L_0}.  
\label{high1}
 \end{align}
Then, we have
\begin{equation}
\begin{split}
\bigg| \sum_{\substack{n, n_1 \in \Z^2\\ n = n_1 - n_2}} &\iintt_{\tau = \tau_1 - \tau_2} f(\tau_1, n_1) g(\tau_2, n_2) h(\tau, n) d\tau d\tau_1 \bigg| \\
&\les L_1^{\frac 12 +} L_2^{\frac 12 +} L_0^{\frac 14 +} N_2^{0-} \| f \|_{L_{\tau_1}^2 \ell_{n_1}^2} \| g \|_{L_{\tau_2}^2 \ell_{n_2}^2} \| h \|_{L_\tau^2 \ell_n^2}.
\end{split}
\label{high}
\end{equation}
\end{lemma}

We first present the proof of \eqref{bilin3} in this case by assuming Lemma \ref{LEM:high}.
Under $N_0^2 \ll N_2 \sim N_1$ and $L_{\textup{max}} \ges N_2$, 
by  applying Lemma \ref{LEM:high}, we have 
\begin{align*}
\text{LHS of \eqref{bilin3}} &\les \sum_{\substack{ N_1, N_2, N_0\ge 1\\\text{dyadic}}}
 \sum_{\substack{L_1, L_2, L_0 \ge 1\\\text{dyadic}}} L_1^{-\frac 12 - \dl_2} L_2^{-\frac 12 - \dl_2} L_0^{-\frac 12 + \dl_1} \\
&\qquad \times L_1^{\frac 12 +} L_2^{\frac 12 +} L_0^{\frac 14 +} N_2^{0-} \| P_{N_1, L_1} u \|_{L_{t,x}^2} \| P_{N_2, L_2} v \|_{L_{t,x}^2} \| P_{N_0, L_0} w \|_{L_{t,x}^2} \\
&\leq \sum_{\substack{N_1, N_2, N_0\ge 1\\\text{dyadic}}} \sum_{\substack{ L_1, L_2, L_0\ge 1\\\text{dyadic}}} L_1^{0-} L_2^{0-} L_0^{0-} N_2^{0-} \| u \|_{L_{t,x}^2} \| v \|_{L_{t,x}^2} \| w \|_{L_{t,x}^2} \\
&\les \| u \|_{L_{t,x}^2} \| v \|_{L_{t,x}^2} \| w \|_{L_{t,x}^2},
\end{align*}
as desired.

We now present the proof of Lemma \ref{LEM:high}.

\begin{proof}[Proof of Lemma \ref{LEM:high}]
By the Cauchy-Schwarz inequality in $\tau$ and $n$, Lemma \ref{LEM:bilin}, and the condition $L_{\textup{max}} \ges N_2 \gg N_0^2$, we have
\begin{align*}
\text{LHS of \eqref{high}} &\les \| h \|_{L_\tau^2 \ell_n^2} \underline{L}_{12}^{\frac 12} 
\bigg( \frac{\overline{L}_{12}}{N_0} + 1 \bigg)^{\frac 12} N_0^{\frac 12} \| f \|_{L_{\tau_1}^2 \ell_{n_1}^2} \| g \|_{L_{\tau_2}^2 \ell_{n_2}^2} \\
&\les \big( L_1^{\frac 12} L_2^{\frac 12} + \underline{L}_{12}^{\frac 12} N_2^{\frac 14} \big) \| f \|_{L_{\tau_1}^2 \ell_{n_1}^2} \| g \|_{L_{\tau_2}^2 \ell_{n_2}^2} \| h \|_{L_\tau^2 \ell_n^2} \\
&\les L_1^{\frac 12 +} L_2^{\frac 12 +} L_0^{\frac 14 +} N_2^{0-} \| f \|_{L_{\tau_1}^2 \ell_{n_1}^2} \| g \|_{L_{\tau_2}^2 \ell_{n_2}^2} \| h \|_{L_\tau^2 \ell_n^2}, 
\end{align*}

\noi
yielding \eqref{high}.
\end{proof}

\medskip \noi
$\circ$  \textbf{Subcase 2.b:} $L_{\textup{max}} \ll N_2$ (low-modulation case).

Our goal in this subcase is to show the following lemma.
\begin{lemma}\label{LEM:low}
Let $N_1, N_2, L_0, L_1, L_2 \geq 1$ be dyadic numbers with $N_1 \sim N_2$ and $L_{\textup{max}} \ll N_2$. Suppose that $f, g, h \in L^2(\R \times \Z^2)$ are real-valued nonnegative functions such that
\begin{align*}
&\textup{supp } f \subset \Pf_{N_1} \cap \Sf_{L_1}, \quad \textup{supp } g \subset \Pf_{N_2} \cap \Sf_{L_2}, 
\quad \text{and}\\
&\textup{supp } h \subset \{n \in \Z^2: 1 \leq |n|^2 \ll N_2 \} \cap \Sf_{L_0}.
\end{align*}
Then, we have 
\begin{equation}
\begin{split}
\bigg| \sum_{\substack{n, n_1\in \Z^2\\ n = n_1 - n_2}} &\iintt_{\tau = \tau_1 - \tau_2} f(\tau_1, n_1) g(\tau_2, n_2) h(\tau, n) \cdot \ind_{|\cos \angle(n, n_2)| \ll \ta} \, d\tau d\tau_1 \bigg| \\
&\les L_{\textup{med}}^{\frac 38} L_{\textup{max}}^{\frac 38} \| f \|_{L_{\tau_1}^2 \ell_{n_1}^2} \| g \|_{L_{\tau_2}^2 \ell_{n_2}^2} \| h \|_{L_\tau^2 \ell_n^2},
\end{split}
\label{low}
\end{equation}
where $\ta = \frac{1}{|n|^{1-\eps} |n_2|^{1-\eps}} \ll 1$.
\end{lemma}

We first assume Lemma \ref{LEM:low} and prove \eqref{bilin3}. 
Recall that  $|n|^2 \ll |n_2|\sim |n_1|$ in this case. 
Then,  by Lemma \ref{LEM:low} (with $\ta$ as above) and the Cauchy-Schwarz inequality in $N_1 \sim N_2$, we have
\begin{align*}
& \text{LHS of } \eqref{bilin3}\\ 
&\hphantom{X} \les \sum_{\substack{N_1 \sim N_2\ge 1\\\text{dyadic}}} 
\sum_{\substack{ L_1, L_2, L_0\ge 1\\\text{dyadic}}} L_1^{-\frac 12 - \dl_2} L_2^{-\frac 12 - \dl_2} L_0^{-\frac 12 + \dl_1} 
\bigg| \sum_{\substack{n, n_1\in \Z^2\\ n = n_1 - n_2}} \iintt_{\tau = \tau_1 - \tau_2}
 \ind_{|\cos \angle (n, n_2)| \ll \ta} \\
&\hphantom{XXX}
\times \ind_{|n|^2 \ll N_2} \ft{P_{N_1, L_1} u} (\tau_1, n_1) \cj{\ft{P_{N_2, L_2} v}} (\tau_2, n_2) \cj{\ft{Q_{L_0} w}} (\tau, n) d\tau_1 d\tau \bigg| \\
&\hphantom{X}
\les  \sum_{\substack{N_1 \sim N_2\ge 1\\\text{dyadic}}} 
\sum_{\substack{ L_1, L_, L_0\ge 1\\\text{dyadic}}}
 L_1^{0-} L_2^{0-} L_0^{0-} \big\| \ft{P_{N_1, L_1} u} \big\|_{L_{\tau_1}^2 \ell_{n_1}^2} \big\| \ft{P_{N_2, L_2} v} \big\|_{L_{\tau_2}^2 \ell_{n_2}^2} \big\| \ft{Q_{L_0} w} \big\|_{L_\tau^2 \ell_n^2} \\
&\hphantom{X}
\les \| \ft u \|_{L_{\tau_1}^2 \ell_{n_1}^2} \| \ft v \|_{L_{\tau_2}^2 \ell_{n_2}^2} \| \ft w \|_{L_\tau^2 \ell_n^2},
\end{align*}

\noi
as desired.

We now present the proof of Lemma \ref{LEM:low}, 
where we use the idea  from the proofs of  of Propositions 3.6 and 3.9 in \cite{Kish13}.

\begin{proof}[Proof of Lemma \ref{LEM:low}]

We first consider the case when $|n|\les L_{\text{max}}$.
 In this case, we decompose the spatial frequencies of $h$ on the left-hand side  of \eqref{low} into dyadic blocks $\{|n| \sim N_0\}$, where the dyadic number  $N_0 \ge 1$ satisfies  $N_0^2 \ll N_2$.
  For  fixed $N_0$, we decompose the spatial frequencies of $f$ and $g$ into  balls of radius $\sim N_0$,
  indexed by $J_1 \in \mathcal{J}_1$ and $J_2 \in \mathcal{J}_2$, respectively.
  With a slight abuse of notation, 
  we also use $J_1$ and $J_2$ to denote the balls themselves.
  Note that for each fixed $J_1 \in \mathcal{J}_1$, the product $\ind_{J_1} (\tau_1, n_1)\cdot  \ind_{J_2} (\tau_1 - \tau, n_1 - n)$ is nonzero for at most $O(1)$ many $J_2 \in \mathcal{J}_2$, and
  we denote the set of these indices $J_2$'s as $\mathcal{J}_2 (J_1)$.

Noting that 
\[
\big| |n_1|^2 - |n_2|^2 \big| = \big| |n|^2 -  (\tau + |n|^2) + (\tau_1 + |n_1|^2) - (\tau_2 + |n_2|^2) \big| \les N_0^2 + L_{\textup{max}},
\]

\noi
we have 
\begin{align}
\big| |n_1| - |n_2| \big| \les N_2^{-1}  (N_0^2 + L_{\textup{max}})  \ll 1.
\label{C1}
\end{align}

 \noi
 Given $j \in \N$, let $A_j = \{ ( \tau_1, n_1) : j < |n_1| \leq j+1 \}$
 and  $B_j = \{ ( \tau_2, n_2) : j  < |n_2| \leq j + 1 \}$.
 In view of \eqref{C1}, we see that  
  the product $\ind_{A_{j_1}} (\tau_1, n_1) \cdot \ind_{B_{j_2}} (\tau_2, n_2)$
   is nonzero if and only if $|j_1 - j_2| \le 1$.
Hence, by letting $f_{j_1,  J_1}(\tau_1, n_1) = \ind_{A_{j_1} \cap J_1} (\tau_1, n_1)\cdot f(\tau_1, n_1)$
and $g_{j_2, J_2}(\tau_2, n_2) = 
 \ind_{B_{j_2} \cap J_2} (\tau_2, n_2)\cdot g(\tau_2, n_2)$, 
 we can estimate  
 the left-hand side of \eqref{low}  by 
\begin{align*}
\sum_{\substack{N_0\ge 1\\ \text{dyadic}}} \sum_{\substack{J_1 \in \mathcal{J}_1 \\ J_2 \in \mathcal{J}_2 (J_1)}} \sum_{\substack{j_1, j_2\in \N\\ |j_1 - j_2| \le 1}} 
\bigg| 
\sum_{\substack{n, n_1 \in \Z^2\\ n = n_1 - n_2\\|n| \sim N_0}} &  \iintt_{\tau = \tau_1 - \tau_2}  f_{j_1, J_1} (\tau_1, n_1)
g_{j_2, J_2} (\tau_2, n_2) 
\\
&\times 
h(\tau, n) \cdot \ind_{S_{\tau, n, j_1, j_2, J_1, J_2}}(\tau_1, n_1) d\tau d\tau_1 \bigg|,
\end{align*}

\noi
where the set $S_{\tau, n, j_1, j_2, J_1, J_2}$ is defined by
\begin{align}
\begin{split}
S_{\tau, n, j_1, j_2, J_1, J_2} := \big\{ &(\tau_1, n_1) \in \R \times \Z^2: 
(\tau_1, n_1) \in   \Pf_{N_1} \cap \Sf_{L_1} \cap A_{j_1},  \, n_1 \in J_1 \\
&(\tau_1 - \tau, n_1 - n) \in  \Pf_{N_2} \cap \Sf_{L_2}\cap B_{j_2}, \, 
n_1 - n \in J_2, \\
&  |\cos \angle (n, n_1 - n)| \ll \ta \big\}.
\end{split}
\label{C1a}
\end{align}

By the Cauchy-Schwarz inequality in $\tau_1$ and $n_1$, H\"older's inequality in $\tau$ and $n$, and H\"older's inequality in $j$, we then have
\begin{align*}
&\text{LHS of }\eqref{low} \\
&\leq \sum_{\substack{N_0\ge 1\\ \text{dyadic}}} \sum_{\substack{J_1 \in \mathcal{J}_1\\ J_2 \in \mathcal{J}_2 (J_1)}} \sum_{\substack{j_1, j_2\in \N\\ |j_1 - j_2| \le 1}} 
 \sum_{n: |n| \sim N_0} \int 
\| \ind_{S_{\tau, n, j_1, j_2, J_1, J_2}}\|_{L^2_{\tau_1} \l^2_{n_1}}\\ 
&\qquad \times 
\| f_{j_1, J_1} (\tau_1, n_1) g_{j_2, J_2} (\tau_1 - \tau, n_1 - n)\|_{L^2_{\tau_1} \l^2_{n_1}}
 |h(\tau, n)| d\tau \\
&\leq \sum_{\substack{N_0\ge 1\\ \text{dyadic}}} \sum_{\substack{J_1 \in \mathcal{J}_1 \\ J_2 \in \mathcal{J}_2 (J_1)}} \sum_{\substack{j_1, j_2\in \N\\ |j_1 - j_2| \le 1}} 
 \underline{L}_{12}^{\frac 12}  \| f _{j_1, J_1}  \|_{L_{\tau_1}^2 \l_{n_1}^2}
  \| g_{j_2, J_2} \|_{L_{\tau_2}^2 \l_{n_2}^2} \| h \|_{L_\tau^2 \l_n^2} \\
&\qquad \times \sup_{\substack{\tau, \tau_1, n\\ |n| \sim N_0}} \big|\wt S_{\tau, \tau_1, n, j_1, j_2, J_1, J_2}\big|^{\frac 12} \\
&\les 
\sum_{\substack{N_0\ge 1\\ \text{dyadic}}} \sum_{\substack{J_1 \in \mathcal{J}_1 \\ J_2 \in \mathcal{J}_2 (J_1)}} \underline{L}_{12}^{\frac 12} \| \ind_{J_1}\cdot f  \|_{L_{\tau_1}^2 \l_{n_1}^2} \|  \ind_{J_2} \cdot g\|_{L_{\tau_2}^2 \l_{n_2}^2} \| h \|_{L_\tau^2 \l_n^2} \\
&\qquad \times \sup_{\substack{\tau, \tau_1, n, j_1, j_2\\ |n| \sim N_0}}  \big|\wt S_{\tau, \tau_1, n, j_1, j_2, J_1, J_2}\big|^{\frac 12} , 
\end{align*}

\noi
where the set $\wt S_{\tau, \tau_1, n, j_1, j_2,  J_1, J_2}$ is defined by
\begin{align}
\wt S_{\tau, \tau_1, n, j_1, j_2,  J_1, J_2} := &\big\{ n_1 \in \Z^2: (\tau_1, n_1) \in 
S_{\tau, n, j_1, j_2, J_1, J_2} \big\}.
\label{C1b}
\end{align}

We claim that the following bound holds:
\begin{align}
\sup_{\substack{J_1 \in \mathcal{J}_1\\J_2 \in \mathcal{J}_2(J_1)}}
 \sup_{\tau, \tau_1, n, j: |n| \sim N_0} \big| \wt  S_{\tau, \tau_1, n, j_1, j_2, J_1, J_2}\big| \les \min\bigg( N_0,  \frac{L_{\textup{max}}}{N_0} \bigg). 
 \label{C2}
\end{align}

\noi
For now, let us assume \eqref{C2}.
Note that the right-hand side of \eqref{C2} is bounded by $L_{\max}^\frac{1}{2}$.
Then, by applying  the Cauchy-Schwarz inequality in $J_1 \in \mathcal{J}_1$ and
summing over dyadic $N_0 \geq 1$ together with \eqref{C2}, 
we  obtain the desired bound \eqref{low}.

It remains to prove \eqref{C2}.
For $n_1 \in \wt  S_{\tau, \tau_1, n, j_1, j_2, J_1, J_2}$, 
  it follows from \eqref{C1b} with \eqref{C1a} that 
\[ \tau_1 \in \big( - |n_1|^2 + [-C L_1, C L_1] \big) \cap \big( \tau - |n_1 - n|^2 + [-C L_2, C L_2] \big), \]
for some constant $C > 0$. Since this intersection of  the two sets is nonempty,
we must have $\big| -|n_1|^2 - (\tau - |n_1 - n|^2) \big| = O(\overline{L}_{12})$, 
which in turn implies 
\begin{align}
 n_1 \cdot \frac{n}{|n|} = -\frac{\tau}{2|n|} + \frac{|n|}{2} + O\bigg( \frac{\overline{L}_{12}}{N_0} \bigg). 
\label{C3}
 \end{align}

\noi
On the other hand,   
for $n_1 \in \wt  S_{\tau, \tau_1, n, j_1, j_2, J_1, J_2}$, 
we also have
$n_1 \in J_1$ where $J_1$ is a ball of radius $\sim N_0$.
Hence, together with \eqref{C3}, we see that  the component of $n_1$ which is parallel to $n$ is restricted to an interval of length $\sim \min(N_0, L_{\text{max}} / N_0)$. Furthermore, we have
\begin{align*}
\begin{split}
 |\cos \angle (n_1, n)| & = \frac{|n_1 \cdot n|}{|n_1| | n|} 
 \leq \frac{|n|}{|n_1|}  + \frac{|(n_1 - n)\cdot n|}{|n_1||n|}\\
 &  \leq \frac{|n|}{|n_1|} + |\cos \angle (n_1 - n, n)| \frac{|n_1 - n|}{|n_1|} \ll \frac{|n|}{|n_1|} + \frac{|n_1 - n|^\eps}{|n|^{1-\eps} |n_1|} \ll 1, 
 \end{split}
\end{align*}

\noi
since $|n_1|\sim |n_1 - n|\gg |n|^2$.
Hence, by Lemma \ref{LEM:count} with $\mu  =  1$, 
 $\nu \sim  \min(N_0, L_{\textup{max}} / N_0)$, and $\al = \frac \pi 4$, we obtain
 \eqref{C2}.

The case  $L_{\text{max}} \ll |n|$ is similar and much simpler. We apply the same steps as in the previous case ($L_{\max}\ges |n|$), except that we do not need to decompose the spatial frequencies of $n$ into dyadic piece or  to localize $n_1$ and $n_2$ on balls of smaller radii.
In this case,  we  apply Lemma \ref{LEM:count} with $\nu \sim 1$ to obtain the desired bound.
\end{proof}

\medskip \noi
$\bullet$ \textbf{Case 3:} $|n_2|^2 \ll |n|$.

This case is similar to Case 2, so we will be brief here. For the low-modulation case (i.e.~$L_{\text{max}} \ll N_0$), the same argument works by switching the roles of $(\tau, n)$ and $(\tau_2, n_2)$. 

For the high-modulation case (i.e.~$L_{\text{max}} \ges N_0$), we need to ensure that on the 
right-hand side of \eqref{high}  in Lemma \ref{LEM:high}, the power of $L_0$ is less than $\frac 12$ (so that simply switching the roles of $(\tau, n)$ and $(\tau_2, n_2)$ does not work). With $f$, $g$, and $h$ as in the statement of Lemma~\ref{LEM:high}, when $N_2^2 \ll N_0 \sim N_1$ and $L_{\text{max}} \ges N_0$, we use the Cauchy-Schwarz inequality in $\tau$ and $n$ and apply Lemma \ref{LEM:bilin} to obtain
\begin{align*}
\text{LHS of } \eqref{high} &\les \| h \|_{L_\tau^2 \ell_n^2}
 \underline{L}_{12}^{\frac 12} \bigg( \frac{\overline{L}_{12}}{N_0} + 1 \bigg)^{\frac 12 } N_2^{\frac 12} \| f \|_{L_{\tau_1}^2 \ell_{n_1}^2} \| g \|_{L_{\tau_2}^2 \ell_{n_2}^2} \\
&\les \big( L_1^{\frac 12} L_2^{\frac 12} + \underline{L}_{12}^{\frac 12} N_0^{\frac 14} \big) \| f \|_{L_{\tau_1}^2 \ell_{n_1}^2} \| g \|_{L_{\tau_2}^2 \ell_{n_2}^2} \| h \|_{L_\tau^2 \ell_n^2} \\
&\les \big( L_1^{\frac 12} L_2^{\frac 12} L_{\text{max}}^{0+} + \underline{L}_{12}^{\frac 12} L_{\text{max}}^{\frac 14 +} \big) N_0^{0-} \| f \|_{L_{\tau_1}^2 \ell_{n_1}^2} \| g \|_{L_{\tau_2}^2 \ell_{n_2}^2} \| h \|_{L_\tau^2 \ell_n^2},
\end{align*}

\noi
which suffices for our purpose.

\begin{remark}\label{REM:bilin}
\rm
The bilinear estimate \eqref{bilin1} in Proposition \ref{PROP:bilin}
also holds if we replace $u \cj{v}$ by $uv$.
In fact, 
a slight modification of the argument allows us to show 
\begin{equation}
\| u v \|_{X_T^{0,-\frac 12 + \dl_1}} \les \| u \|_{X_T^{0, \frac 12 + \dl_2}} \| v \|_{X_T^{0, \frac 12 + \dl_2}},
\label{bilin_uv}
\end{equation}

\noi
where $0 < T \leq 1$ and $\dl_1 >  \dl_2 > 0$ are sufficiently small. 
Note that, in this case, by denoting $n_1$, $n_2$, and $n$ as the frequency of $u$, $v$, and 
the duality term $w$, respectively.  we have $n = n_1 + n_2$ and $\tau = \tau_1 + \tau_2$ so that 
\[ (\tau + |n|^2) - (\tau_1 + |n_1|^2) - (\tau_2 + |n_2|^2) = 2n_1 \cdot n_2. \]

\noi
Hence, we need to perform case-by-case analysis, depending on the interaction between $n_1$ and $n_2$.

The proof of \eqref{bilin_uv} follows essentially as in the proof of \eqref{bilin1} except for Lemma \ref{LEM:high}. With $f$, $g$, and $h$ as in the statement of Lemma \ref{LEM:high}, when $N_1^2 \ll N_2 \sim N_0$ and $L_{\text{max}} \ges N_2$,  the Cauchy-Schwarz inequality in $\tau_2$ and $n_2$ and
 Lemma \ref{LEM:bilin} yield 
\begin{equation}
\begin{split}
\bigg| \sum_{\substack{n, n_1\in \Z^2\\ n = n_1 + n_2}} &
\iintt_{\tau = \tau_1 + \tau_2} f(\tau_1, n_1) g(\tau_2, n_2) h(\tau, n) d\tau d\tau_1 \bigg| \\
&\les \big( L_0^{\frac 12} L_1^{\frac 12} N_2^{-\frac 14} + \underline{L}_{01}^{\frac 12} N_2^{\frac 14} \big) \| f \|_{L_{\tau_1}^2 \ell_{n_1}^2} \| g \|_{L_{\tau_2}^2 \ell_{n_2}^2} \| h \|_{L_\tau^2 \ell_n^2}.
\end{split}
\label{uv1}
\end{equation}

\noi
In order to repeat the argument in Subcase 2.a in the proof of Proposition \ref{PROP:bilin}, 
we need the power of $L_0$ to be less than $\frac 12$, especially  when $L_{\text{max}} = L_0$. 
Namely, the bound \eqref{uv1} can not be used directly as it is.

By H\"older's inequality in $\tau$ and $n$, Young's inequality, and H\"older's inequality
with \eqref{high1}, we obtain
\begin{equation}
\begin{split}
\bigg| \sum_{\substack{n, n_1\in \Z^2\\ n = n_1 + n_2}} &
\iintt_{\tau = \tau_1 + \tau_2} f(\tau_1, n_1) g(\tau_2, n_2) h(\tau, n) d\tau d\tau_1 \bigg| \\
&\les \| f \|_{\ell_{n_1}^2 L_{\tau_1}^{\frac 8 7}} \| g \|_{\ell_{n_2}^1 L_{\tau_2}^2} \| h \|_{\ell_n^2 L_{\tau}^{\frac 85}} \\
&\les L_0^{\frac 18} L_1^{\frac 38} N_2 \| f \|_{L_{\tau_1}^2 \ell_{n_1}^2} \| g \|_{L_{\tau_2}^2 \ell_{n_2}^2} \| h \|_{L_\tau^2 \ell_n^2}.
\end{split}
\label{uv2}
\end{equation}

\noi
Then, by interpolating \eqref{uv1} and \eqref{uv2}, we  obtain the desired inequality. A similar inequality also holds  when $N_2^2 \ll N_1 \sim N_0$ and $L_{\text{max}} \ges N_1$ by switching the roles of $(\tau_1, n_1)$ and $(\tau_2, n_2)$.

%
\end{remark}

\section{Local well-posedness of the quadratic NLS}
\label{SEC:LWP}

In this section, we present the proof of Theorem \ref{THM:LWP}, i.e. local well-posedness of the quadratic NLS \eqref{qNLS} in $L^2(\T^2)$.
By writing \eqref{qNLS} in the following Duhamel formulation, we have 
\begin{equation}
u(t) = \G(u) := e^{it\Dl} u_0 - i \int_0^t e^{i(t-t')\Dl} |u|^2(t') dt'.
\label{Duh}
\end{equation}

\noi
Let $\eps > 0$ be sufficiently small. Then, by \eqref{Duh}, Lemma \ref{LEM:Xlin}, and Proposition \ref{PROP:bilin}, we have
\begin{align*}
\| \G(u) \|_{X_T^{0, \frac 12 + \eps}} &\leq \| e^{it\Dl} u_0 \|_{X_T^{0, \frac 12 + \eps}} + \bigg\| \int_0^t e^{i(t-t')\Dl} u(t') \cj{u}(t') dt' \bigg\|_{X_T^{0, \frac 12 + \eps}} \\
&\les \| u_0 \|_{L^2} + T^{\eps }\| u \cj{u} \|_{X^{0, -\frac 12 + 2 \eps}} \\
&\les \| u_0 \|_{L^2} + T^{\eps } \| u \|_{X_T^{0, \frac 12 + \eps}}^2.
\end{align*}

\noi
 Similarly, we  obtain the following difference estimate:
\begin{align*}
\| \G(u) - \G(v) \|_{X_T^{0, \frac 12 + \eps}}
\les T^{ \eps}\Big( \| u \|_{X_T^{0, \frac 12 + \eps}} + \| v \|_{X_T^{0, \frac 12 + \eps}} \Big)
 \| u-v \|_{X_T^{0, \frac 12 + \eps}} .
\end{align*}

\noi
Therefore,  by choosing $T = T(\| u_0 \|_{L^2}) > 0$
sufficiently small, we conclude  that $\G$ is a contraction on the ball $B_R \subset X^{0, \frac 12 + \eps}_T$ of radius $R \sim \| u_0 \|_{L^2}$.
This proves  Theorem~\ref{THM:LWP}.

\begin{ackno}\rm
R.L.~and T.O.~were supported by the European Research Council (grant no.~864138 ``SingStochDispDyn'').
 They are also grateful to the anonymous referee for the helpful comments.

\end{ackno}

\end{document}